\documentclass[12 pt]{amsart}
\usepackage{amssymb,latexsym,amsmath,amscd,amsthm,graphicx, color}
\usepackage[all]{xy}
\usepackage{pgf,tikz}
\usepackage{mathrsfs}
\usepackage{cite}

\usetikzlibrary{arrows}
\definecolor{uuuuuu}{rgb}{0.26666666666666666,0.26666666666666666,0.26666666666666666}
\definecolor{xdxdff}{rgb}{0.49019607843137253,0.49019607843137253,1.}
\definecolor{ffqqqq}{rgb}{1.,0.,0.}

\definecolor{uuuuuu}{rgb}{0.26666666666666666,0.26666666666666666,0.26666666666666666}
\definecolor{qqwuqq}{rgb}{0.,0.39215686274509803,0.}
\definecolor{zzttqq}{rgb}{0.6,0.2,0.}
\definecolor{xdxdff}{rgb}{0.49019607843137253,0.49019607843137253,1.}
\definecolor{qqqqff}{rgb}{0.,0.,1.}
\definecolor{cqcqcq}{rgb}{0.7529411764705882,0.7529411764705882,0.7529411764705882}

\definecolor{uuuuuu}{rgb}{0.26666666666666666,0.26666666666666666,0.26666666666666666}
\definecolor{qqwuqq}{rgb}{0.,0.39215686274509803,0.}
\definecolor{zzttqq}{rgb}{0.6,0.2,0.}
\definecolor{xdxdff}{rgb}{0.49019607843137253,0.49019607843137253,1.}
\definecolor{qqqqff}{rgb}{0.,0.,1.}
\definecolor{cqcqcq}{rgb}{0.7529411764705882,0.7529411764705882,0.7529411764705882}

\theoremstyle{plain}

\newtheorem{theorem}[subsection]{Theorem}

\newtheorem{lemma}[subsection]{Lemma}

\newtheorem{rem}[subsection]{Remark}
\newtheorem{prop}[subsection]{Proposition}
\newtheorem{defi}[subsection]{Definition}

\theoremstyle{definition}

\newtheorem{cor}[subsection]{Corollary}

\newtheorem{remark}[subsection]{Remark}

\newcommand{\uu}{\cup}
\newcommand{\ii}{\cap}
\newcommand{\UU}{\bigcup}

\newcommand{\ci}{\subseteq}
\newcommand{\sci}{\subset}
\newcommand{\es}{\emptyset}
\newcommand{\set}[1]{\{#1\}}

\newcommand{\ga}{\alpha}
\newcommand{\gb}{\beta}

\newcommand{\gd}{\delta}
\renewcommand{\gg}{\gamma}

\newcommand{\go}{\omega}

\newcommand{\gs}{\sigma}

\newcommand{\tit}{\textit}

\newcommand{\D}[1]{\mathbb{#1}}

\newcommand{\te}{\text}

\begin{document}

To appear, Houston Journal of Mathematics
\title{The quantization of the standard triadic Cantor distribution}

\author{Mrinal Kanti Roychowdhury}
\address{School of Mathematical and Statistical Sciences\\
University of Texas Rio Grande Valley\\
1201 West University Drive\\
Edinburg, TX 78539-2999, USA.}
\email{mrinal.roychowdhury@utrgv.edu}

\thanks{ }

\date{}
\maketitle

\pagestyle{myheadings}\markboth{Mrinal Kanti Roychowdhury}{The quantization of the standard triadic Cantor distribution}

\begin{abstract} The quantization scheme in probability theory deals with finding a best approximation of a given probability distribution by a probability distribution that is supported on finitely many points. For a given $k\geq 2$, let $\{S_j : 1\leq j\leq k\}$ be a set of $k$ contractive similarity mappings such that $S_j(x)=\frac 1 {2k-1} x +\frac{2 (j-1)} {2k-1}$ for all $x\in \mathbb R$, and let  $P= \frac 1 k  \sum_{j=1}^kP\circ S_j^{-1}$. Then, $P$ is a unique Borel probability measure on $\mathbb R$ such that $P$ has support the Cantor set generated by the similarity mappings $S_j$ for $1\leq j\leq k$. In this paper, for the probability measure $P$, when $k=3$, we investigate the optimal sets of $n$-means and the $n$th quantization errors for all $n\geq 2$. We further show that the quantization coefficient does not exist though the quantization dimension exists.
\end{abstract}

\section{Introduction}
One of the main mathematical aims of the quantization problem is to study the
error in the approximation of a given probability measure with a probability measure of finite support. We refer to
\cite{GL1, GL3, GL4, GL5, P} for more theoretical results, and \cite{P1, P2} for promising applications of quantization theory.
One may see \cite{GG, GN, Z} for its deep background in information theory and engineering technology.
Let $\D R^d$ denote the $d$-dimensional Euclidean space, $\|\cdot\|$ denote the Euclidean norm on $\D R^d$ for any $d\geq 1$, and $n\in \D N$. For a finite set $\ga \sci \D R^d$, the number $\int \min_{a \in \ga} \|x-a\|^2 dP(x)$ is often referred to as the \tit{cost} or \tit{distortion error} for $\ga$, and is denoted by $V(P; \ga)$. Then, the $n$th quantization error, denoted by $V_n:=V_n(P)$, is defined by
\begin{equation*} \label{eq0} V_n=\inf \Big\{V(P; \ga) : \alpha \subset \mathbb R^d, \text{ card}(\alpha) \leq n \Big\}.\end{equation*}
Such a set $\ga$ for which the infimum occurs and contains no more than $n$ points is called an \tit{optimal set of $n$-means}, and is denoted by $\ga_n:=\ga_n(P)$.
 It is known that for a continuous probability measure an optimal set of $n$-means always has exactly $n$ elements (see \cite{GL1}). To see some work in the direction of optimal sets of $n$-means, one is referred to \cite{CR, DR1, GL2,  R1, R2, R3, R4, R5, RR}. The number
\[ D(P):=\lim_{n\to \infty}  \frac{2\log n}{-\log V_n(P)},\]
 if it exists, is called the \tit{quantization dimension} of $P$ and is denoted by $D(P)$. For any $s\in (0, +\infty)$, the number
 \[\lim_{n\to \infty}n^{\frac 2 s} V_n(P),\] if it exists, is called the $s$-dimensional \tit{quantization coefficient} for $P$.
Given a finite subset $\ga\sci \D R^d$,  the Voronoi region generated by $a\in \ga$ is the set of all elements in $\D R^d$ which are closer to $a$ than to any other element in $\ga$.
 Let us now state the following proposition (see \cite{GG, GL1}).
\begin{prop} \label{prop0}
Let $\ga$ be an optimal set of $n$-means, $a \in \alpha$, and $M(a|\ga)$ be the Voronoi region generated by $a\in \ga$, i.e.,
$M(a|\ga)=\{x \in \mathbb R^d : \|x-a\|=\min_{b \in \alpha} \|x-b\|\}.$
Then, for every $a \in\alpha$,
$(i)$ $P(M(a|\ga))>0$, $(ii)$ $ P(\partial M(a|\ga))=0$, $(iii)$ $a=E(X : X \in M(a|\ga))$.
\end{prop}
From the above proposition, we can say that if $\ga$ is an optimal set of $n$-means for $P$, then each $a\in \ga$ is the conditional expectation of the random variable $X$ given that $X$ takes values in the Voronoi region of $a$. Sometimes, we also refer to such an $a\in \ga$ as the centroid of its own Voronoi region.
In this regard, interested readers can see \cite{DFG, DR2, R1}.

Let $k\geq 2$ be a positive integer, and let $\set{S_j : 1 \leq j\leq k}$ be a set of contractive similarity mapping such that $S_j(x)=\frac 1 {2k-1} x +\frac{2 (j-1)} {2k-1}$ for all $x\in \D R$, and let  $P=\frac 1 k\sum_{j=1}^k  P\circ S_j^{-1}$. Then, $P$ is a unique Borel probability measure on $\D R$, and $P$ has support the Cantor set $C$ generated by the similarity mappings $S_j$ for $1\leq j\leq k$, and $C$ satisfies the invariance equality $C=\mathop{\uu}\limits_{j=1}^k S_j(C)$  (see \cite{F, H}). The Cantor set $C$ generated by the $k$ similarity mappings is called the \tit{$k$-adic Cantor set},  more specifically the \tit{standard $k$-adic Cantor set}, and the probability measure $P$ is called the \tit{$k$-adic Cantor distribution}, more specifically the \tit{standard $k$-adic Cantor distribution}.  If $k=2$, then we have two similarity mappings given by $S_1(x)=\frac 1 3 x$ and $S_2(x)=\frac 13 x +\frac 23$ for all $x\in \D R$, and then the probability measure $P$ is given by $P=\frac 12 P\circ S_1^{-1}+\frac 12 P\circ S_2^{-1}$, which has support the classical Cantor set $C$ satisfying $C=S_1(C)\uu S_2(C)$. For this dyadic Cantor distribution, in \cite{GL2}, Graf and Luschgy determined the optimal sets of $n$-means and the $n$th quantization errors for all $n\geq 2$. They also showed that the quantization dimension $D(P)$ of $P$ exists, and equals the Hausdorff dimension of the invariant set $C$, but the quantization coefficient does not exist.

 By a word $\gs$ of length $n$, where $n\geq 1$, over the alphabet $\{1, 2, 3\}$, it is meant that $\gs:=\gs_1\gs_2\cdots \gs_n$, where $\gs_j\in \set{1, 2, 3}$ for all $1\leq j\leq n$. By $\{1, 2, 3\}^n$, we denote the set of all words
over the alphabet $\{1, 2, 3\}$ of some finite length $n\geq 0$. Notice that the empty word $\es$ has length zero. For $\gs:=\gs_1\gs_2 \cdots\gs_n \in \{ 1, 2, 3\}^n$, by $S_\gs$ it is meant that $S_\gs:=S_{\gs_1}\circ \cdots \circ S_{\gs_n}$. For the empty word $\es$, by $S_\es$ it is meant the identity mapping on $\D R$. Write \[\set{1, 2, 3}^\ast:=\mathop{\uu}\limits_{n=0}^\infty \set{1, 2, 3}^n,\]
i.e., $\set{1, 2, 3}^\ast$ denotes the set of all words over the alphabet $\set{1, 2, 3}$ including the empty word $\es$.
Let $X$ be a random variable with probability distribution $P$. For words $\gb, \gg, \cdots, \gd$ in $\set{1,2, 3}^\ast$, by $a(\gb, \gg, \cdots, \gd)$ we mean the conditional expectation of the random variable $X$ given $J_\gb\uu J_\gg \uu\cdots \uu J_\gd,$ i.e.,
\begin{equation*} \label{eq45} a(\gb, \gg, \cdots, \gd)=E(X|X\in J_\gb \uu J_\gg \uu \cdots \uu J_\gd)=\frac{1}{P(J_\gb\uu \cdots \uu J_\gd)}\int_{J_\gb\uu \cdots \uu J_\gd} x dP(x).
\end{equation*}

\begin{defi}\label{defi0} For $n\in \D N$ with $n\geq 3$ let $\ell(n)$ be the unique natural number with $3^{\ell(n)} \leq n< 3^{\ell(n)+1}$. Write
$\ga_2:=\set{a(1, 21), a(22, 23, 3)}$ and $\ga_3:=\set{a(1), a(2), a(3)}$. For $n\geq 3$, define $\ga_n:=\ga_n(I)$ as follows: \[\ga_n(I)=\left\{\begin{array}{cc}
\set{a(\go) : \go \in \set{1, 2, 3}^{\ell(n)}\setminus I}\UU \mathop{\uu}\limits_{\go\in I} S_\go(\ga_2) & \te{ if } 3^{\ell(n)}\leq n\leq 2\cdot 3^{\ell(n)},\\
 \set{S_\go(\ga_2) : \go \in \set{1, 2, 3}^{\ell(n)}\setminus I}\UU \mathop{\uu}\limits_{\go\in I} S_\go(\ga_3) & \te{ if } 2\cdot 3^{\ell(n)}< n< 3^{\ell(n)+1},
\end{array}
\right.\] where $I \sci \set{1, 2, 3}^{\ell(n)}$ with $\te{card}(I)=n-3^{\ell(n)}$ if $3^{\ell(n)}\leq n\leq 2\cdot3^{\ell(n)}$; and
$\te{card}(I)=n-2\cdot 3^{\ell(n)}$ if $2\cdot 3^{\ell(n)}< n< 3^{\ell(n)+1}$.
\end{defi}

In this paper, in Section~\ref{sec2} we show that for all $n\geq 2$, the sets $\ga_n$ given by Definition~\ref{defi0} form the optimal sets of $n$-means for the standard triadic Cantor distribution $P$. In Section~\ref{sec3}, we show that the quantization coefficient does not exist though the quantization dimension $D(P)$ exists. Notice that the probability measure $P$ is symmetric about the point $\frac 12$, i.e., if two intervals of equal lengths are equidistant from the point $\frac 12$, then they have the same probability. Thus, it seems, which is true in the case of dyadic Cantor distribution (see \cite{GL2}), that if the closed interval $[0, 1]$ is partitioned in the middle, then the conditional expectations of the left half $[0, \frac 12]$, and the right half $[\frac 12, 1]$ will form the optimal set of two-means. In Proposition~\ref{prop0001}, we show that it is not true in the case of triadic Cantor distribution. The result in this paper extends the well-known result for the dyadic Cantor distribution given by Graf-Luschgy, and we are grateful to say that the work in this paper was motivated by their work (see \cite{GL2}).

\section{Preliminaries} \label{sec1}
Let $S_j$ for $1\leq j\leq 3$ be the contractive similarity mappings on $\D R$ given by $S_j(x)=\frac 1 {5} x +\frac 2 5(j-1)$ for all $x\in \D R$. For $\gs:=\gs_1\gs_2 \cdots\gs_n \in \{ 1, 2, 3\}^n$, set $J_\gs:=S_{\gs}([0, 1])$, where  $S_\gs:=S_{\gs_1}\circ \cdots \circ S_{\gs_n}$. For the empty word $\es$, write $J:=J_\es=S_\es([0,1])=[0, 1]$. Then, the set $C:=\bigcap_{n\in \mathbb N} \bigcup_{\gs \in \set{1, 2, 3}^n} J_\gs$ is known as the \textit{Cantor set} generated by the mappings $S_j$, and equals the support of the probability measure $P$ given by $P=\sum_{j=1}^3 \frac 1 3 P\circ S_j^{-1}$. For $\gs=\gs_1\gs_2 \cdots\gs_n \in \{ 1, 2, 3\}^\ast$, $n\geq 0$, write
$p_\gs:=\frac 1{3^n}$ and $s_\gs:=\frac 1 {5^n}$.

Let us now give the following lemmas. The proofs are similar to the similar lemmas in \cite{GL2}.

\begin{lemma} \label{lemma1}
Let $f : \mathbb R \to \mathbb R^+$ be Borel measurable and $k\in \mathbb N$. Then,
\[\int f dP=\sum_{\sigma \in \{1, 2,3\}^k} \frac 1 {3^k} \int (f \circ S_\sigma)(x) dP(x).\]
\end{lemma}

\begin{lemma} \label{lemma2} Let $X$ be a random variable with probability distribution $P$. Then,
$E(X)=\frac 12  \text{ and } V:=V(X)=\frac 19,$ and for any $x_0 \in \mathbb R$,
$\int (x-x_0)^2 dP(x) =V +(x_0-\frac 12)^2.$
\end{lemma}

From Lemma~\ref{lemma1} and Lemma~\ref{lemma2} the following corollary is true.
\begin{cor} \label{cor1}
Let $\gs \in \set{1, 2, 3}^\ast$ and $x_0 \in \mathbb R$. Then,
\begin{equation}  \int_{J_\gs} (x-x_0)^2 dP(x) =p_\gs\Big(s_\gs^2 V  +(S_\gs(\frac 12)-x_0)^2\Big).\end{equation}
\end{cor}

\begin{remark}
From the above lemma it follows that the optimal set of one-mean is the expected value and the corresponding quantization error is the variance $V$ of the random variable $X$.
For $\gs \in \{1, 2, 3\}^n$, $n\geq 1$, we have $a(\gs)=E(x : X\in J_\gs)=S_\gs(\frac 12)$.
\end{remark}

\begin{prop}\label{prop00}
Let $\ga_n:=\ga_n(I)$ be the set given by Definition~\ref{defi0}. If $3^{\ell(n)}\leq n\leq 2\cdot 3^{\ell(n)}$, then the number of such sets is  ${}^{3^{\ell(n)}}C_{n-3^{\ell(n)}}$, and the corresponding distortion error is given by
\[V(P; \ga_n(I))=\frac 1{75^{\ell(n)}} \left((2\cdot 3^{\ell(n)}-n)V+ (n-3^{\ell(n)}) V(P; \ga_2)\right).\]
 If $2\cdot 3^{\ell(n)}< n< 3^{\ell(n)+1}$, then the number of such sets is $ {}^{3^{\ell(n)}}C_{n- 2\cdot 3^{\ell(n)}}$, and the corresponding distortion error is given by
\[V(P; \ga_n(I))=\frac 1{75^{\ell(n)}} \left((3^{\ell(n)+1}-n)V(P; \ga_2)+ (n-2\cdot 3^{\ell(n)})V (P; \ga_3)\right).\]
 \end{prop}
\begin{proof} If $3^{\ell(n)}\leq n\leq 2\cdot 3^{\ell(n)}$, then the subset $I$ can be chosen in ${}^{3^{\ell(n)}}C_{n-3^{\ell(n)}}$ different ways, and so, the number of such sets is given by ${}^{3^{\ell(n)}}C_{n-3^{\ell(n)}}$, and the corresponding distortion error is obtained as
\begin{align*}
&V(P; \ga_n(I))=\sum_{\gs \in \set{1, 2, 3}^{\ell(n)}\setminus I}\int_{J_\gs}(x-a(\gs))^2 dP+\sum_{\gs \in I} \int_{J_\gs}\min_{a\in S_\gs(\ga_2)}(x-a)^2 dP\\
&=\frac 1{3^{\ell(n)}}\frac 1{25^{\ell(n)}}\Big( \sum_{\gs \in \set{1, 2, 3}^{\ell(n)}\setminus I}  V+\sum_{\gs \in I}  V(P; \ga_2)\Big)=\frac 1 {75^{\ell(n)}}\Big ((2\cdot 3^{\ell(n)}-n) V+(n-3^{\ell(n)}) V(P; \ga_2)\Big).
\end{align*}
Similarly, the other part of the proposition can be derived. Thus, the proof of the proposition is complete.
\end{proof}
The following proposition is helpful to find the optimal set of two-means given in Lemma~\ref{lemmaB00}. It also shows that though the triadic Cantor distribution is uniform and symmetric about the point $\frac 12$, the set consisting of the conditional expectations of the left half $[0, \frac 12]$, and the right half $[\frac 12, 1]$ does not form an optimal set of two-means.

\begin{prop} \label{prop0001}
Let $a_1:=E(X : X\in [0, \frac 12])$, and $a_2:=E(X : X\in [\frac 12, 1])$. Then, the set $\gg:=\set{a_1, a_2}$ does not form an optimal set of two-means for $P$. \end{prop}
\begin{proof}
By the hypothesis, we have
\begin{align*}
a_1&=E(X : X \in [0, \frac 12])=E\Big(X : X\in  J_1\uu J_{21}\uu J_{221}\uu \cdots \Big), \te{ and } \\
a_2&=E(X : X \in [\frac 12,  1])=E\Big(X : X\in  J_3\uu J_{23}\uu J_{223}\uu \cdots \Big),
\end{align*}
yielding
\[a_1=2 \sum_{n=1}^\infty \frac 1 {3^n}\frac 12 \frac  {5^n-4}{5^n} = \frac{3}{14}, \te{ and }  a_2=2 \sum_{n=1}^\infty \frac 1 {3^n}\frac 12 \frac  {5^n+4}{5^n}=\frac{11}{14}.\]
 and the corresponding distortion error is given by
\begin{align*}
& V(P; \gg) = 2 \int_{J_1\uu J_{21}\uu J_{221}\uu J_{2221}\cdots} \Big(x-\frac 3{14}\Big)^2 dP
\end{align*}
implying \[V(P; \gg)=2 \Big( \sum _{n=1}^{\infty } \frac{1}{75^n} V+\sum _{n=1}^{\infty } \frac 1 {3^n} \Big(\frac 12 \frac  {5^n-4}{5^n}-\frac 3{14}\Big)^2 \Big)=\frac{13}{441}.\]
Let us now consider the set $\gb:=\set{a(1, 21), a(22, 23, 3)}$. Since $\frac{11}{25}=S_{21}(1)<\frac 12 (a(1, 21)+ a(22, 23, 3))=\frac{117}{250}<S_{22}(0)=\frac{12}{25}$, the distortion error due to the set $\gb$ is given by
\[V(P; \gb)=\int_{J_1\uu J_{21}}(x-a(1, 21))^2 dP+\int_{J_{22}\uu J_{23}\uu J_3}(x-a(21, 22, 23, 3))^2 dP=\frac{821}{28125}.\]
Since $V(P; \gg)=\frac{13}{441}>\frac{821}{28125}=V(P; \gb)$, the set $\gg$ does not form an optimal set of two-means yielding the proposition.
\end{proof}

\section{Optimal sets of $n$-means and the $n$th quantization errors for all $n\geq 2$} \label{sec2}

Let $\ga_n$ be the set given by Definition~\ref{defi0}. In this section, we show that for all $n\geq 2$, the sets $\ga_n$ form the optimal sets of $n$-means for the standard triadic Cantor distribution $P$. To calculate the quantization error we will frequently use the formula given by \eqref{eq0}.

\begin{lemma} \label{lemmaB00}
The set $ \ga_2:=\set{a(1, 21), a(22, 23, 3)}$ forms an optimal set of two-means, and the corresponding quantization error is given by
$V_{2}=\frac{821}{28125}.$
\end{lemma}
\begin{proof}
Consider the set $\gb$ of two points given by $\gb:=\set{a(1, 21), a(22, 23, 3)}$. The distortion error due to the set $\gb$ is given by
\[\int \min_{b\in \gb}(x-b)^2 dP=\int_{J_1\uu J_{21}}(x-a(1, 21))^2 dP+\int_{J_{22}\uu J_{23}\uu J_3}(x-a(22, 23, 3))^2 dP=\frac{821}{28125}.\]
Since $V_2$ is the quantization error for two-means, we have $V_2\leq \frac{821}{28125}$. Let $\ga:=\set{a, b}$ be an optimal set of two-means. Since the optimal quantizers are the expected values of their own Voronoi regions, we have $0<a<b<1$. By Proposition~\ref{prop0001}, we see that the boundary of the Voronoi regions can not pass through the midpoint $\frac 12$. Thus, without any loss of generality we assume that the boundary of the Voronoi regions, i.e., the point $\frac 12 (a+b)$ lies to the left of the midpoint $\frac 12$, i.e., $\frac 12(a+b)<\frac 12$. We now show that the Voronoi region of $a$ contains points from $J_2$. For the sake of contradiction, assume that the Voronoi region of $a$ does not contain any point from $J_2$. Then,
\[V_2\geq \int_{J_2\uu J_3}(x-a(2, 3))^2 dP=\frac{4}{135}>V_2,\]
which leads to a contradiction. Hence, we can assume that the Voronoi region of $a$ contains points from $J_2$, i.e.,  $S_2(0)=\frac 25<\frac 12(a+b)<\frac 12$. Assume that $S_2(0)=\frac 25<\frac 12(a+b)<S_{2112}(0)$. Then, writing $A_1:=J_{2112}\uu J_{2113}\uu J_{212}\uu J_{213}\uu J_{22}\uu J_{23}\uu J_3$, we have
\[V_2\geq \int_{J_1}(x-a(1))^2 dP+\int_{A_1}(x-a(2112, 2113, 212, 213, 22, 23, 3))^2 dP=\frac{452551888}{15092578125}>V_2,\]
which leads to a contradiction. Hence, $S_{2112}(0)<\frac 12(a+b)<\frac 12$. Assume that $S_{2112}(0)<\frac 12(a+b)<S_{2113}(0)$. Then, writing $A_2:=J_{2113}\uu J_{212}\uu J_{213}\uu J_{22}\uu J_{23}\uu J_3$, we have
\[V_2\geq \int_{J_1\uu J_{2111}}(x-a(1, 2111))^2 dP+\int_{A_2}(x-a(2113, 212, 213, 22, 23, 3))^2 dP=\frac{775266524}{25913671875}>V_2,\]
which gives a contradiction. Hence, $S_{2113}(0)<\frac 12(a+b)<\frac 12$. Assume that $S_{2113}(0)<\frac 12(a+b)<S_{212}(0)$.
Then, writing $A_3:=J_{212}\uu J_{213}\uu J_{22}\uu J_{23}\uu J_3$, we have
\[V_2\geq \int_{J_1\uu J_{2111}\uu J_{2112}}(x-a(1, 2111, 2112))^2 dP+\int_{A_3}(x-a(212, 213, 22, 23, 3))^2 dP=\frac{4180404512}{140389453125}\]
yielding $V_2>V_2$, which is a contradiction. Hence,  $S_{212}(0)<\frac 12(a+b)<\frac 12$.
Assume that $S_{212}(0)<\frac 12(a+b)<S_{2122}(0)$. Then, writing $A_4:=J_{2122}\uu J_{2123}\uu J_{213}\uu J_{22}\uu J_{23}\uu J_3$, we have
\[V_2\geq \int_{J_1\uu J_{211}}(x-a(1, 211))^2 dP+\int_{A_4}(x-a(2122, 2123, 213, 22, 23, 3))^2 dP=\frac{210852992}{7119140625}>V_2,\]
which is a contradiction. Hence, $S_{2122}(0)<\frac 12(a+b)<\frac 12$. Assume that $S_{2122}(0)<\frac 12(a+b)<S_{2123}(0)$.
Then, writing $A_5:=J_{2123}\uu J_{213}\uu J_{22}\uu J_{23}\uu J_3$, we have
\[V_2\geq \int_{J_1\uu J_{211}\uu J_{2121}}(x-a(1, 211, 2121))^2 dP+\int_{A_5}(x-a(2123, 213, 22, 23, 3))^2 dP=\frac{12733641776}{432558984375}\]
yielding $V_2>V_2$, which is a contradiction. Hence, $S_{2123}(0)<\frac 12(a+b)<\frac 12$. Assume that $S_{2123} (0)<\frac 12(a+b)<S_{213}(0)$.
Writing $A_6=J_1\uu J_{211}\uu J_{2121}\uu J_{2122}$, and $A_7=J_{213}\uu J_{22}\uu J_{23}\uu J_3$, we have
\[V_2\geq \int_{A_6}(x-a(1, 211, 2121, 2122))^2 dP+\int_{A_7}(x-a(213, 22, 23, 3))^2 dP=\frac{332546}{11390625}>V_2,\]
which give a contradiction. Hence, $S_{213}(0)<\frac 12(a+b)<\frac 12$. Assume that $S_{213}(0)<\frac 12(a+b)<S_{21313}(0)$. Then, writing
$A_8=J_1\uu J_{211}\uu J_{212}$, $a_8=E(X : X \in A_8)$, $A_9=J_{21313}\uu J_{2132}\uu J_{2133}\uu J_{22}\uu J_{23}\uu J_3$, and $a_9=E(X : X\in A_9)$, we have
\[V_2\geq \int_{A_8}(x-a_8)^2 dP+\int_{A_9}(x-a_9)^2 dP=\frac{489226058779}{16680146484375}>V_2,\]
which leads to a contradiction. Hence, $S_{21313}(0)<\frac 12(a+b)<\frac 12$. Assume that $S_{21313}(0)<\frac 12(a+b)<S_{2132}(0)$. Then, writing
$A_{10}=J_1\uu J_{211}\uu J_{212}\uu J_{21311}\uu J_{21312}$, $a_{10}=E(X : X \in A_{10})$, $A_{11}=J_{2132}\uu J_{2133}\uu J_{22}\uu J_{23}\uu J_3$, and $a_{11}=E(X : X\in A_{11})$, we have
\[V_2\geq \int_{A_{10}}(x-a_{10})^2 dP+\int_{A_{11}}(x-a_{11})^2 dP=\frac{2996202402374}{101383681640625}>V_2,\]
which leads to a contradiction. Hence, $S_{2132}(0)<\frac 12(a+b)<\frac 12$.

Assume that $S_{2132}(0)<\frac 12(a+b)<S_{2133}(0)$. Partition the interval
$[S_{2132}(0), S_{2133}(0)]$ into the following three subintervals:
\[[S_{21321}(0), S_{21322}(0)], \, [S_{21322}(0), S_{21323}(0)], \, [S_{21323}(0), S_{2133}(0)].\]
Then, $\frac 12(a+b)$ belongs to one of the above three subintervals. First, assume that $\frac 12(a+b) \in [S_{21321}(0), S_{21322}(0)]$. Then, as before we can show that the distortion error is larger than $0.0291911\geq V_2$, which leads to a contradiction. To show that for $\frac 12(a+b) \in [S_{21321}(0), S_{21322}(0)]$, the distortion error is larger, we may need to partition the subinterval $[S_{21321}(0), S_{21322}(0)]$ into the following three sub-subintervals:
\[[S_{213211}(0), S_{213212}(0)], \, [S_{213212}(0), S_{213213}(0)], [S_{213213}(0), S_{21322}(0)],\] and then we can separately consider the cases as follows:
\[\frac 12(a+b)\in [S_{213211}(0), S_{213212}(0)], \, [S_{213212}(0), S_{213213}(0)], \te{ or }  [S_{213213}(0), S_{21322}(0)].\]
If needed, we can further partition the above sub-subintervals to check that the distortion is larger. Similarly, we can show that if $\frac 12(a+b)$ belongs to either $[S_{21322}(0), S_{21323}(0)]$, or $[S_{21323}(0), S_{2133}(0)],$ then the contradiction arises. Therefore, $S_{21313}(0)<\frac 12(a+b)<S_{2132}(0)$ can not happen. Using the similar arguments, we can show that neither $S_{2132}(0)<\frac 12(a+b)<S_{2133}(0)$, nor $S_{2133}(0)<\frac 12(a+b)<S_{2133}(1)$ can  happen. Hence, we can assume that $S_{21}(1)<\frac 12(a+b)<\frac 12$.

Assume that $S_{22222}(0)<\frac 12(a+b)<\frac 12$. Then, writing $A_{12}=J_1\uu J_{21}\uu J_{221}\uu J_{2221}\uu J_{22221}$, and $a_{12}=E(X : X \in A_{12})$, and by  Proposition~\ref{prop0001}, noting the fact that  $\int_{[\frac 12, 1]}(x-E(X : X\in [\frac 12, 1]))^2 dP=\frac {13}{882}$, we have
\[V_2\geq \int_{A_{12}}(x-a_{12})^2 dP+\frac {13}{882}=\frac{61346429393}{2093027343750}>V_2,\]
which gives a contradiction. Hence, $S_{21}(1)<\frac 12(a+b)<S_{22222}(0)$. Assume that $S_{222212}(0)<\frac 12(a+b)<S_{22222}(0)$. Then,
writing $A_{13}=J_1\uu J_{21}\uu J_{221}\uu J_{2221}\uu J_{222211}$,   $a_{13}=E(X : X \in A_{13})$, $A_{14}=J_{22222}\uu J_{22223}\uu J_{2223}\uu J_{223}\uu J_{23}\uu J_3$, and $a_{14}=E(X : X \in A_{14})$, we have
\[V_2\geq \int_{A_{13}}(x-a_{13})^2 dP+\int_{A_{14}}(x-a_{14})^2 dP=\frac{1031786636229481}{35273384033203125}>V_2,\]
which is a contradiction. Hence, $S_{21}(1)<\frac 12(a+b)<S_{222212}(0)$. Proceeding in this way, and using the similar technique of partitioning the intervals into subintervals, as mentioned in the previous paragraph, we can show that $S_{22}(0)<\frac 12(a+b)<S_{222212}$ cannot happen.
Hence, $S_{21}(1)<\frac 12(a+b)<S_{22}(0)$. Thus, the set $ \ga_2:=\set{a(1, 21), a(22, 23, 3)}$ forms an optimal set of two-means, and the corresponding quantization error is given by
$V_{2}=\frac{821}{28125}.$
Hence, the proof of the lemma is complete.
\end{proof}

\begin{cor} \label{cor0}
Let $\ga_2$ be an optimal set of two-means. Then, for $1\leq j\leq 3$, we have $\int_{J_j}\min_{a\in S_j(\ga_2)}(x-a)^2 dP=\frac 1{75} V_2$.
\end{cor}
\begin{proof}
We have \begin{align*} &\int_{J_j}\min_{a\in S_j(\ga_2)}(x-a)^2dP=\frac 13 \int_{J_j}\min_{a\in S_j(\ga_2)}(x-a)^2d(P\circ S_j^{-1})=\frac 13 \int\min_{a\in S_j(\ga_2) }(S_j(x)-a)^2dP\\
&=\frac 13 \int\min_{a\in \ga_2 }(S_j(x)-S_j(a))^2dP=\frac 1{75}  \int\min_{a\in \ga_2 }(x-a)^2dP=\frac 1{75} V_2.
\end{align*}
Thus, the proof of the corollary is complete.
\end{proof}

\begin{lemma} \label{lemmaB1}
The set $\ga_3:=\set{S_1(\frac 12), S_2(\frac 12), S_3(\frac 12)}$ forms an optimal set of three-means, and the corresponding quantization error is given by $V_{3}=\frac{1}{225}$.
\end{lemma}
\begin{proof}
Let $\gb$ be a set of three points such that $\gb:=\set{S_j(\frac 12) : j=1, 2, 3}$. Then,
\[\int\min_{b\in \gb} (x-b)^2 dP=\sum_{j=1}^3 \int_{J_j} (x-S_j(\frac 1 2))^2 dP=\frac{1}{225}.\]
Since $V_3$ is the quantization error for three-means, we have $V_3\leq\frac{1}{225}$. Let $\ga:=\set{a_1, a_2, a_3}$ be an optimal set of three-means. Since the elements in an optimal set are the centroids of their own Voronoi regions, without any loss of generality, we can assume that $0<a_1<a_2<a_3<1$. We now prove that $\ga\ii J_j\neq \es$ for all $1\leq j\leq 3$. Suppose that $\ga\ii J_1=\es$. Then,
\[V_3\geq \int_{J_1}(x-\frac 15)^2 dP=\frac{13}{2700}>V_3,\]
which is a contradiction. So, we can assume that $\ga\ii J_1\neq \es$. Similarly, $\ga\ii J_3\neq \es$. We now that that $\ga\ii J_2\neq \es$.
Suppose that $\ga\ii J_2=\es$. Then, either $a_2<\frac 25$, or $a_2>\frac 35$. Assume that $a_2<\frac 25$. Then, as $\frac 12(\frac 25+\frac 45)=\frac 35$, we have
\[V_3\geq \int_{J_2}(x-\frac 25)^2 dP+\int_{J_3}(x-S_3(\frac 12))^2dP=\frac{17}{2700}>V_3,\]
which is a contradiction. Similarly, we can show that if $\frac 35<a_2$, then a contradiction arises. Thus, we can assume that $a_2\in J_2$, i.e., $\ga\ii J_2\neq \es$. Now, if the Voronoi region of $a_1$ contains points from $J_2$, we have $\frac 12(a_1+a_2)>\frac 25$ implying $a_2>\frac 45-a_1\geq \frac 45-\frac 15=\frac 35$, which is a contradiction as $a_2\in J_2$. Thus, the Voronoi region of $a_1$ does not contain any point from $J_2$. Similarly, the Voronoi region of $a_2$ does not contain any point from $J_1$ and $J_3$. Likewise, the Voronoi region of $a_3$ does not contain any point from $J_2$. Since the optimal quantizers are the centroids of their own Voronoi regions, we have $a_1=S_1(\frac 12)$, $a_2=S_2(\frac 12)$, and $a_3=S_3(\frac 12)$, and the corresponding quantization error is given by $V_3=\frac{1}{225}$. Thus, the proof of the lemma is complete.
\end{proof}

\begin{remark}\label{remark0}
By Lemma~\ref{lemmaB1}, we see that the set $\set{S_j(\frac 12) : 1\leq j\leq 3}$ forms an optimal set of three-means. Similarly, we can show that the sets
 $\set{S_1(\frac 12), S_2(\frac 12)}\uu S_3(\ga_2) $, $\set{S_1(\frac 12 )}\uu S_2(\ga_2)\uu S_3(\ga_2)$, $S_1(\ga_2)\uu S_2(\ga_2)\uu S_3(\ga_2)$,  $S_{1}(\ga_3)\uu S_2(\ga_2)\uu S_3(\ga_2)$,  and $S_1(\ga_3)\uu S_2(\ga_3)\uu S_3(\ga_2)$  form optimal sets of $n$-means for $n=4, 5, 6, 7, 8$, respectively. Due to technicality of the proofs, we do not show them in the paper.

\end{remark}

\begin{prop}\label{propB1}
Let $\ga_n$ be an optimal set of $n$-means for any $n\geq 3$. Then, $\ga_n\ii J_j\neq \es$ for all $1\leq j\leq 3$, and $\ga_n$ does not contain any point from the open intervals $(\frac 15, \frac 25)$ and $(\frac 35, \frac 45)$. Moreover, the Voronoi region of any point in $\ga_n\ii J_j$ does not contain any point from $J_i$, where $1\leq i\neq j\leq 3$.
\end{prop}

\begin{proof}
Due to Lemma~\ref{lemmaB1}, and Remark~\ref{remark0}, the proposition is true for $3\leq n\leq 8$. Let us now prove that the proposition is true for $n\geq 9$. Let $\ga_n:=\set{a_1, a_2, \cdots, a_n}$ be an optimal set of $n$-means for $n\geq 9$. Since the points in an optimal set are the centroids of their own Voronoi regions, without any loss of generality, we can assume that $0<a_1<a_2<\cdots<a_n<1$. Consider the set of nine elements $\gb:=\set{S_\gs(\frac 12) : \gs \in\set{1, 2, 3}^2}$. Then,
\[\int\min_{a\in \gb}(x-a)^2 dP=\sum_{\gs\in \set{1,2,3}^2}\int_{J_\gs}(x-a(\gs))^2 dP=\frac 1{25^2} V =\frac{1}{5625}.\]
Since $V_n$ is the quantization error for $n$-means for $n\geq 9$, we have $V_n\leq V_9\leq \frac{1}{5625}$. Suppose that $\frac 15< a_1$. Then,
\[V_n\geq \int_{J_1}(x-\frac 15)^2 dP=\frac{13}{2700}>V_n,\]
which is a contradiction. So, we an assume that $a_1\leq \frac 15$. Similarly, $\frac 45\leq a_n$. Thus, $\ga_n\ii J_1\neq \es$, and $\ga_n\ii J_3\neq \es$. We now show that $\ga_n\ii J_2\neq \es$. For the sake of contradiction, assume that $\ga_n\ii J_2=\es$. Let $a_j:=\max\set{a_i : a_i<\frac 25 \te{ for } 1\leq i\leq n-1}$. Then, $a_j<\frac 25$. As $\ga_n\ii J_2=\es$, we have $\frac 35<a_{j+1}$.
Thus, using the symmetry and the formula given by \eqref{eq0}, we have
\begin{align*}
&V_n\geq 2 \sum_{n=2}^\infty \int_{J_{2^{n-1}1}}(x-\frac 25)^2 dP =2 \sum_{n=2}^\infty  \frac 1 {3^n}\Big(\frac 1{25^n} V+(S_{2^{n-1}1}(\frac 12)-\frac25)^2\Big)\\
&=2 \sum_{n=2}^\infty  \left(\frac 1{75^n} V+\frac 1 {3^n} \Big((S_{2^{n-1}1}(\frac 12))^2-\frac 45 S_{2^{n-1}1}(\frac 12)+\frac 4{25}\Big)\right)\\
&=2 \sum_{n=2}^\infty  \frac 1 {75^n}V +2 \sum_{n=2}^\infty  \frac 1 {3^n}\Big (\frac {5^n-4}{5^{n-1} 10}\Big)^2-\frac 8 5 \sum_{n=2}^\infty  \frac 1 {3^n}  \frac {5^n-4}{5^{n-1} 10}+\frac 8{25} \sum_{n=2}^\infty \frac 1 {3^n}\\
&=\frac{19}{18900}>V_n,
\end{align*}
which gives a contradiction. Hence, we can conclude that $\ga_n\ii J_2\neq \es$.
Next, suppose that $\ga_n$ contains a point from the open interval $(\frac 15, \frac 25)$. Let $a_j:=\max\set{a_i : a_i<\frac 15 \te{ for } 1\leq i\leq n-2}$. Then, $a_{j+1}\in (\frac 15, \frac 25)$, and $a_{j+2}\in J_2$. The following cases can arise:

Case~1. $\frac 15 <a_{j+1}\leq \frac 3{10}$.

Then, $\frac12(a_{j+1}+a_{j+2})>\frac 25$ implying $a_{j+2}>\frac 45-a_{j+1}\geq \frac 45-\frac 3{10}=\frac 12$ implying, as before,
\begin{align*} V_n&\geq \sum_{n=2}^\infty \int_{J_{2^{n-1}1}}(x-\frac 12)^2 dP\\
&= \sum_{n=2}^\infty  \frac 1 {75^n}V + \sum_{n=2}^\infty  \frac 1 {3^n}\Big (\frac {5^n-4}{5^{n-1} 10}\Big)^2- \sum_{n=2}^\infty  \frac 1 {3^n}  \frac {5^n-4}{5^{n-1} 10}+\frac 14\sum_{n=2}^\infty \frac 1 {3^n}\\
&=\frac{1}{1350}>V_n,
\end{align*}
which leads to a contradiction.

Case~2. $  \frac 3{10}\leq a_{j+1}<\frac 25$.

Then, $\frac 12(a_j+a_{j+1})<\frac 15$ implying $a_j\leq \frac 25-a_{j+1}\leq \frac 25-\frac 3{10}=\frac 1{10}$. Then,
\[V_n\geq \int_{J_{13}}(x-\frac 1{10})^2 dP=\frac{37}{50625}>V_n,\]
which yields a contradiction.

Thus, by Case~1 and Case~2, we can conclude that $\ga_n$ does not contain any point from the open interval $(\frac 15, \frac 25)$. Reflecting the situation with respect to the point $\frac 12$, we can conclude that $\ga_n$ does not contain any point from the open interval $(\frac 35, \frac 45)$ as well. To prove the last part of the proposition, we proceed as follows:  $a_j=\max\set{a_i : a_i<\frac 15 \te{ for } 1\leq i\leq n-2}$. Then, $a_j$ is the rightmost element in $\ga_n\ii J_1$, and $a_{j+1}\in \ga_n\ii J_2$. Suppose that the Voronoi region of $a_j$ contains points from $J_2$. Then, $\frac 12(a_j+a_{j+1})>\frac 25$ implying $a_{j+1}>\frac 45-a_j\geq \frac 45-\frac 15=\frac 35$, which yields a contradiction as $a_{j+1}\in J_2$. Thus, the Voronoi region of any point in $\ga_n\ii J_1$ does not contain any point $J_2$, and so from $J_3$ as well. Similarly, we can prove that the Voronoi region of any point in $\ga_n\ii J_2$ does not contain any point from $J_1$ and $J_3$, and the Voronoi region of any point in $\ga_n\ii J_3$ does not contain any point from $J_1$ and $J_2$. Thus, the proof of the proposition is complete.
\end{proof}

The following lemma is similar to Lemma~4.5 in \cite{GL2}.
\begin{lemma} \label{lemmaB4}
Let $n\geq 3$, and let $\ga_n$ be an optimal set of $n$-means such that $\ga_n\ii J_j\neq \es$ for all $1\leq j\leq 3$, and $\ga_n$ does not contain any point from the open intervals $(\frac 15, \frac 25)$ and $(\frac 35, \frac 45)$. Further assume that the Voronoi region of any point in $\ga_n\ii J_j$ does not contain any point from $J_i$, where $1\leq i\neq j\leq 3$. Set $\gb_j:=\ga_n\ii J_j$, and $n_j:=\te{card }(\gb_j)$ for $1\leq j\leq 3$.
Then, $S_j^{-1}(\gb_j)$ is an optimal set of $n_j$-means, and
$V_{n}=\frac 1{75}(V_{n_1}+V_{n_2}+V_{n_3}).$
\end{lemma}

\begin{proof} By the hypothesis, $\gb_j \neq \emptyset ,$ for all $1\leq j\leq 3 , $ and $\ga_n$ does not contain any point from the open intervals $(\frac 15, \frac 25)$ and $(\frac 35, \frac 45)$. Hence, $  \ga_n=\mathop{\uu}\limits_{j=1}^3 \gb_j$.   Since $\ga_n$  is an optimal set of $n$-means,
\begin{equation*} \label{eq46}
V_n=\sum_{j=1}^3\int_{J_j} \min_{a\in \ga_n} \|x-a\|^2 dP=\sum_{j=1}^3\int_{J_j} \min_{a\in \gb_j} \|x-a\|^2 dP.
\end{equation*}
Now, using Lemma~\ref{lemma1} we have
\begin{equation}\label{eq47}
V_n=\frac{1}{75} \sum_{j=1}^3\int \min_{a\in \gb_j} \|x-S_j^{-1}(a)\|^2 dP=\frac{1}{75} \sum_{j=1}^3\int \min_{a\in S_j^{-1}(\gb_j)} \|x-a\|^2 dP.
\end{equation}
If $S_1^{-1}(\gb_1)$ is not an optimal set of $n_1$-means, then we can find a set $\gg_1\sci \D R^2$ with card$(\gg_1)=n_1$ such that
\[\int \min_{a\in \gg_1} \|x-a\|^2 dP <\int \min_{a\in S_1^{-1}(\gb_1)} \|x-a\|^2 dP.\]
But, then $S_1(\gg_1)\uu \gb_2\uu \gb_3$ will be a set of cardinality $n$, and
\begin{align*}
&\int \min\set{ \|x-a\|^2 : a\in S_1(\gg_1)\uu \gb_2\uu \gb_3}dP\\
&=\int_{J_1} \min_{a\in S_1(\gg_1)} \|x-a\|^2 dP+\frac{1}{75} \sum_{j=2}^3\int \min_{a\in S_j^{-1}(\gb_j)} \|x-a\|^2 dP\\
&=\frac 1 {75} \int \min_{a\in S_1(\gg_1)} \|x-S_1^{-1}(a)\|^2 dP+\frac{1}{75} \sum_{j=2}^3\int \min_{a\in S_j^{-1}(\gb_j)} \|x-a\|^2 dP\\
&=\frac 1 {75} \int \min_{a\in \gg_1} \|x-a\|^2 dP+\frac{1}{75} \sum_{j=2}^3\int \min_{a\in S_j^{-1}(\gb_j)} \|x-a\|^2 dP\\
&<\frac 1 {75} \int \min_{a\in S_1^{-1}(\gb_1)} \|x-a\|^2 dP+\frac{1}{75} \sum_{j=2}^3\int \min_{a\in S_j^{-1}(\gb_j)} \|x-a\|^2 dP.
\end{align*}
Thus by \eqref{eq47}, we have $\int \min\set{ \|x-a\|^2 : a\in S_1(\gg_1)\uu \gb_2\uu \gb_3}dP<V_n$, which contradicts the fact that $\ga_n$ is an optimal set of $n$-means, and so $S_1^{-1}(\gb_1)$ is an optimal set of $n_1$-means. Similarly, one can show that $S_j^{-1}(\gb_j)$ are optimal sets of $n_j$-means for $j=2, 3$. Thus, \eqref{eq47} implies that $V_n=\frac {1}{75} \left(V_{n_1}+V_{n_2}+V_{n_3}\right)$. This completes the proof of the lemma.
\end{proof}

Let us now state and prove the following theorem which gives the optimal sets of $n$-means for all $n\geq 3$.
\begin{theorem}  \label{Th1}
Let $P$ be the standard triadic Cantor distribution on $\D R$ with support the Cantor set $C$ generated by the three contractive similarity mappings $S_j$ for $j=1,2, 3$. Let $n\in \D N$ with $n\geq 3$. Then, the set $\ga_n:=\ga_n(I)$ given by Definition~\ref{defi0} forms an optimal set of $n$-means for $P$ with the corresponding quantization error $V_n:=V(P;\ga_n(I))$, where $V(P; \ga_n(I))$ is given by Proposition~\ref{prop00}.
\end{theorem}

\begin{proof} We will proceed by induction on $\ell(n)$. If $\ell(n)=1$, then the theorem is true by Remark~\ref{remark0}. Let us assume that the theorem is true for all $\ell(n)<m$, where $m\in \D N$ and $m\geq 2$. We now show that the theorem is true if $\ell(n)=m$. Let us first assume that $3^m\leq n\leq 2\cdot 3^m$.
Let $\ga_n$ be an optimal set of $n$-means for $P$ such that $3^m\leq n\leq 2\cdot 3^m$. Let $\te{card }(\ga_n\ii J_j)=n_j$ for $j=1, 2,3$. Then, by Lemma~\ref{lemmaB4}, we have
\begin{equation} \label{eq341} V_{n}=\frac 1{75}(V_{n_1}+V_{n_2}+V_{n_3}).
\end{equation}
Without any loss of generality, we can assume that $n_1\geq n_2\geq n_3$. Let $p, q, r\in \D N$ be such that
\begin{equation}\label{eq35} 3^p\leq n_1\leq 2\cdot 3^p, \   3^q\leq n_2\leq 2\cdot 3^q, \te{ and } 3^r\leq n_3\leq 2\cdot 3^r.
\end{equation}
We will show that $p=q=r=m-1$. Since $n_1\geq n_2\geq n_3$, we have $n_1\geq 3^{m-1}$, and $n_3\leq 2\cdot 3^{m-1}$ implying $r\leq m-1\leq p$. If $\tilde V_n$ is the distortion error due to the set
 $\set{a(\gs) : \gs \in \set{1, 2, 3}^{m}\setminus I} \uu \left(\uu_{\gs \in I} S_\gs(\ga_2)\right)$, by Proposition~\ref{prop00}, we have
\begin{align*}
\tilde V_n=\frac 1{75^{m}} \left((2\cdot 3^{m}-n)V+ (n-3^{m})V_2\right).
\end{align*}
Thus, by the induction hypothesis, \eqref{eq341} implies $\tilde V_n\geq V_n$ yielding
\begin{align*}
&\frac 1{75^{m}} \left((2\cdot 3^{m}-n)V+ (n-3^{m})V_2\right)\geq\frac 1{75^{p+1}} \left((2\cdot 3^{p}-n_1)V+ (n_1-3^{p})V_2\right)\\
&\qquad +\frac 1{75^{q+1}} \left((2\cdot 3^{q}-n_2)V+ (n_2-3^{q})V_2\right)+\frac 1{75^{r+1}} \left((2\cdot 3^{r}-n_3)V+ (n_3-3^{r})V_2\right)
\end{align*}
which upon simplification gives,
\begin{align} \label{eq76}
&3\left((2V-V_2)-\frac {n}{3^m}(V-V_2)\right)\geq25^{m-p-1}\left((2V-V_2)-\frac {n_1}{3^p}(V-V_2)\right)\\
&\qquad +25^{m-q-1} \left((2V-V_2)-\frac {n_2}{3^q}(V-V_2)\right)+25^{m-r-1}\left((2V-V_2)-\frac {n_3}{3^r}(V-V_2)\right).\notag
\end{align}
Hence, using the bounds of $\frac {n}{3^m}$, $\frac {n_1}{3^p}$, $\frac {n_2}{3^q}$, and $\frac {n_3}{3^r}$, i.e., putting $\frac {n}{3^m}=1$, and $\frac {n_1}{3^p}=\frac {n_2}{3^q}=\frac {n_3}{3^r} =2$, from the above inequality, we obtain
\begin{equation} \label{eq45} 11.419 = \frac {3V}{V_2}\geq 25^{m-p-1} +25^{m-q-1}+25^{m-r-1}.
\end{equation}
Recall that $r\leq m-1\leq p$. Moreover, $p>m$ is not possible.  Thus, from \eqref{eq45}, to obtain the values of $p$, $q$, and $r$, we proceed as follows:

$(i)$ If $p=r=m-1$, then \eqref{eq45} implies that $q=m-1$.

$(ii)$ If $p=m-1$ and $r=m-2$, then $11.419\geq 1+25^{m-q-1}+25$, which gives a contradiction, in fact, a contradiction arises for $p=m-1$ and any $r< m-1$.

$(iii)$ If $p=m$ and $r=m-1$, then \eqref{eq45} implies that $q\geq m-1$. Notice that if $q=m$, then
$n>n_1+n_2\geq 3^m+3^m=2\cdot 3^m$, which is a contradiction. Also, $q>m$ is not possible. So, $q=m-1$ is the only choice, but then also a contradiction arises as shown below:
For $p=m, \, q=m-1$, and $r=m-1$, \eqref{eq76} implies
\begin{align*}
&3\left((2V-V_2)-\frac {n}{3^m}(V-V_2)\right)\geq  (2V-V_2)(\frac 1{25}+2)-(V-V_2)(\frac 1{25} \frac {n_1}{3^m} +\frac {n_2}{3^{m-1}}+\frac {n_3}{3^{m-1}}),
\end{align*}
i.e., \[(2V-V_2) \frac {24}{25} \geq (V-V_2) (\frac n{3^{m-1}}-\frac {n_2}{3^{m-1}}-\frac {n_3}{3^{m-1}}-\frac 1 {25}\frac {n_1}{3^{m}}),\]
which upon simplification yields,
\[\frac {n_1}{3^m}\leq \frac{2V-V_2}{V-V_2} \frac {12}{37}=\frac{5429}{7104}<1,\]
which is a contradiction because $3^p\leq n_1\leq 2\cdot 3^p$, and $p=m$.

$(iv)$ If $p=m$ and $r=m-2$, then \eqref{eq45} implies $11.419\geq \frac 1{25} +25^{m-q-1}+25$, which is not true. In fact, a contradiction arises for $p=m$ and any $r< m-1$.

Hence, we can conclude that $p=q=r=m-1$. Since by Lemma~\ref{lemmaB4}, for $S_j^{-1}(\ga_n\ii J_j)$ is an optimal set of $n_j$ means where $3^{m-1}\leq n_j\leq 2\cdot 3^{m-1}$, we have
\[S_j^{-1}(\ga_n\ii J_j)=\set{a(\gs) : \gs \in \set{1, 2, 3}^{m-1}\setminus I_j} \uu \left(\uu_{\gs \in I_j} S_\gs(\ga_2)\right),\]
where $I_j\ci \set{1, 2, 3}^{m-1}$ with $\te{card }(I_j)=n_j-3^{m-1}$ for $1\leq j\leq 3$. Hence,
\[\ga_n:=\ga_n(I)=\UU_{j=1}^3 S_j^{-1}(\ga_n\ii J_j)=\set{a(\gs) : \gs \in \set{1, 2, 3}^{\ell(n)}\setminus I} \uu \left(\uu_{\gs \in I} S_\gs(\ga_2)\right),\]
where $I\ci \set{1, 2, 3}^{m}$ with $\te{card }(I)=n-3^{m}$, is an optimal set of $n$-means. The corresponding quantization error is
\[V_n=\frac 1{75^{m}} \left((2\cdot 3^{m}-n)V+ (n-3^{m})V_2\right)=V(P; \ga_n(I)),\]
where $V(P; \ga_n(I))$ is given by Proposition~\ref{prop00}. Thus, the theorem is true if $3^m\leq n\leq 2\cdot 3^m$. Similarly, we can prove that the theorem is true if $2\cdot 3^{m}< n< 3^{m+1}$. Hence, by the induction principle, the proof of the theorem is complete.
\end{proof}

\begin{rem}
In Theorem~\ref{Th1}, if $n=3^{\ell(n)}$, then $I$ is the empty set implying $\ga_n:=\ga_n(I)=\set{a(\gs) : \gs\in \set{1, 2, 3}^{\ell(n)}}$, and the corresponding quantization error is given by $V_n=\frac 1{25^{\ell(n)}}V$.
\end{rem}

 \section{Quantization dimension and quantization coefficient} \label{sec3}

Since the Cantor set $C$ under investigation satisfies the strong separation condition, with each $S_j$ having contracting factor of $\frac{1}{5}$, the Hausdorff dimension of the Cantor set is equal to the similarity dimension.  Hence, from the equation $3(\frac 1{5})^\gb=1$,  we have $\dim_{\te{H}}(C)=\beta =\frac {\log 3}{\log 5}$.  By Theorem~14.17 in \cite{GL1}, the quantization dimension $D (P) $ exists and is equal to $\gb$. In this section, we show that  $\gb$ dimensional quantization coefficient for $P$ does not exist.

\begin{lemma} \label{lemmaBB1}
Let $f : [1, 2]\to \D R$ be a function defined by
\[f(x)=x^{\frac 2 \gb}\left((2V-V_2)-x(V-V_2)\right)=\frac 1{28125}(5429 - 2304 x)x^{2/\beta }.\]
Then,

$(i)$ $1< \frac{5429}{1152 (\beta +2)}<2$, and

$(ii)$ $f(\frac{5429}{1152 (\beta +2)}>f(2)>f(1)$, and $f([1, 2])=[f(1), f(\frac{5429}{1152 (\beta +2)})]$.
\end{lemma}
\begin{proof}
Notice that $f$ is a continuous function on the closed interval $[1, 2]$, and
\[f'(x)=\frac{x^{\frac{2}{\beta }-1} (10858-2304 (\beta +2) x)}{28125 \beta }.\]
$f'(x)=0$ implies that $x=\frac{5429}{1152 (\beta +2)}$. Moreover, $f'(1)=-\frac{2 (1152 \beta -3125)}{28125 \beta }>0$, and $f'(2)=\frac{4^{1/\beta } (821-2304 \beta )}{28125 \beta }<0$. Thus, $f$ is maximum at $x=\frac{5429}{1152 (\beta +2)}$. Again, $f(2)=\frac{821}{28125}4^{1/\beta }>\frac 19=f(1)$. Hence, $1< \frac{5429}{1152 (\beta +2)}<2$, and $f(\frac{5429}{1152 (\beta +2)})>f(2)>f(1)$, and $f([1, 2])=[f(1), f(\frac{5429}{1152 (\beta +2)})]$. Thus, the proof of the lemma is complete.
\end{proof}

\begin{theorem}
The $\gb$-dimensional quantization coefficient does not exist.
\end{theorem}

\begin{proof} Let $M=\max\set{f(x) : 1\leq x\leq 2}$. Then, by Lemma~\ref{lemmaBB1}, we have $M=\frac{5429}{1152 (\beta +2)}$.
Let $(n_k)_{k\in \D N}$ be a subsequence of the set of natural numbers such that $3^{\ell(n_k)}\leq n_k<2\cdot 3^{\ell(n_k)}$. The assertion of the theorem will follow if we show that the set of accumulation points of the sequence $(n_k^{\frac 2\gb} V_{n_k})_{k\geq 1}$ is $[f(1), f(M)]$. Let  $y \in [f(1), f(M)], $ then $y=f(x)$ for some $x\in [1, 2]$. Set $n_{k_\ell}=\lfloor x 3^{\ell}\rfloor$, where $\lfloor x 3^{\ell}\rfloor$ denotes the greatest integer less than or equal to $ x 3^{\ell}$. Then, $n_{k_\ell}<n_{k_{\ell+1}}$ and $\ell(n_{k_\ell})=\ell$, where by $\ell(n_{k_\ell})=\ell$ it is meant that $3^\ell\leq n_{k_\ell}<2 \cdot 3^\ell$. Notice that then there exists $x_{k_\ell} \in [1, 2]$ such that $n_{k_\ell}=x_{k_\ell} 3^\ell$. Recall that $  3^{\frac 1 {\gb}}=5$, and if $3^{\ell(n)}\leq n\leq 2\cdot 3^{\ell(n)}$, then by Theorem~\ref{Th1}, we have
\begin{align*}
V_n=\frac 1{75^{\ell(n)}} \left((2\cdot 3^{\ell(n)}-n)V+ (n-3^{\ell(n)})V_2\right).
\end{align*}
Thus, putting the values of $n_{k_\ell}$ and $V_{n_{k_\ell}}$, we obtain
\begin{align*}
n_{k_\ell}^{\frac 2\gb} V_{n_{k_\ell}}&=n_{k_\ell}^{\frac 2\gb} \frac 1{75^\ell} \left((2\cdot 3^\ell-n_{k_\ell})V+ (n_{k_\ell}-3^\ell)V_2\right)
\end{align*}
yielding
\begin{equation} \label{eq451}
n_{k_\ell}^{\frac 2\gb} V_{n_{k_\ell}}=x_{k_\ell}^{\frac 2\gb} \left((2V-V_2) - x_{k_\ell}(V-V_2)\right)=f(x_{k_\ell}).
\end{equation}
Again, $x_{k_\ell}3^{\ell}\leq x3^\ell<x_{k_\ell}3^{\ell}+1$, which implies $x-\frac 1{3^{\ell}}< x_{k_\ell} \leq x$, and so, $\mathop{\lim}\limits_{\ell\to \infty} x_{k_\ell}=x$. Since $f$ is continuous, we have
\[\mathop{\lim}\limits_{\ell\to \infty}  n_{k_\ell}^{\frac 2\gb} V_{n_{k_\ell}}=f(x)=y,\]
which yields the fact that $y$ is an accumulation point of the subsequence $(n_k^{\frac 2\gb}  V_{n_k})_{k\geq 1}$ whenever $y\in [f(1), f(M)]$. To prove the converse, let $y$ be an accumulation point of the sequence  $(n_k^{\frac 2 \gb} V_{n_k})_{k\geq 1}$. Then, there exists a subsequence  $(n_{k_i}^{\frac 2\gb} V_{n_{k_i}})_{i\geq 1}$ of  $(n_k^{\frac 2\gb} V_{n_k})_{k\geq 1}$ such that $\mathop{\lim}\limits_{i\to \infty}n_{k_i}^{\frac 2\gb} V_{n_{k_i}}=y$. Set $\ell_{k_i}=\ell(n_{k_i})$ and $x_{k_i}=\frac{n_{k_i}}{3^{\ell_{k_i}}}$. Then,  $x_{k_i} \in[1, 2]$, and as shown in \eqref{eq451}, we have
\[n_{k_i}^{\frac 2\gb} V_{n_{k_i}}=f(x_{k_i}).\]
Let $(x_{k_{i_j}})_{j\geq 1}$ be a convergent subsequence of $(x_{k_i})_{i\geq 1}$, then we obtain
\[y=\lim_{i\to \infty} n_{k_i}^{\frac 2\gb} V_{n_{k_i}}=\lim_{j\to \infty}n_{k_{i_j}}^{\frac 2\gb}  V_{n_{k_{i_j}}}=\lim_{j\to \infty}f(x_{k_{i_j}}) \in [f(1), f(M)].\]
Thus, the set of accumulation points of the sequence $(n_k^{\frac 2\gb} V_{n_k})_{k\geq 1}$ is $[f(1), f(M)]$, i.e., the $\gb$-dimensional quantization coefficient for $P$ does not exist. Hence, the proof of the theorem is complete.
\end{proof}


\begin{thebibliography}{9999}

 \bibitem[CR]{CR} D. Comez and M.K. Roychowdhury, \emph{Quantization for uniform distributions on stretched Sierpinski triangles}, Monatshefte f\"ur Mathematik, Volume 190, Issue 1, 79-100 (2019).
%

    \bibitem[DFG]{DFG} Q. Du, V. Faber and M. Gunzburger, \emph{Centroidal Voronoi Tessellations: Applications and Algorithms}, SIAM Review, Vol. 41, No. 4 (1999), pp. 637-676.


\bibitem[DR1]{DR1} C.P. Dettmann and M.K. Roychowdhury, \emph{Quantization for uniform distributions on equilateral triangles}, Real Analysis Exchange, Vol. 42(1), 2017, pp. 149-166.

\bibitem[DR2]{DR2} C.P. Dettmann and M.K. Roychowdhury, \emph{An algorithm to compute CVTs for finitely generated Cantor distributions}, to appear, Southeast Asian Bulletin of Mathematics.


\bibitem[F]{F} K.J. Falconer, \emph{Techniques in Fractal Geometry}, Chichester: Wiley, 1997.

\bibitem[GG]{GG} A. Gersho and R.M. Gray, \emph{Vector quantization and signal compression}, Kluwer Academy publishers: Boston, 1992.
%



\bibitem[GL1]{GL1} S. Graf and H. Luschgy, \emph{Foundations of quantization for probability distributions}, Lecture Notes in Mathematics 1730, Springer, Berlin, 2000.


\bibitem[GL2]{GL2} S. Graf and H. Luschgy, \emph{The Quantization of the Cantor Distribution}, Math. Nachr., 183 (1997), pp. 113-133.

\bibitem[GL3]{GL3} S. Graf and H. Luschgy, \emph{Asymptotics of the quantization error for self-similar probabilities}, Real Anal. Exchange 26, 795-810 (2001).

\bibitem[GL4]{GL4} S. Graf and H. Luschgy, \emph{Quantization for probability measures with respect to the geometric mean error}, Math. Proc.
Cambridge Philos. Soc. 136, 687-717 (2004).

\bibitem[GL5]{GL5} S. Graf and H. Luschgy and G. Pag\`{e}s, \emph{The local quantization behavior of absolutely continuous probabilities}, Ann.
Probab. 40, 1795-1828 (2012).

\bibitem[GN]{GN}  R. Gray and D. Neuhoff, \emph{Quantization,} IEEE Trans. Inform. Theory,  44 (1998), pp. 2325-2383.


\bibitem[H]{H} J. Hutchinson, \emph{Fractals and self-similarity}, Indiana Univ. J., 30 (1981), pp. 713-747.


\bibitem[P]{P} K. P\"{o}tzelberger, \emph{The quantization dimension of distributions}, Math. Proc. Cambridge Philos. Soc., Volume 131, Issue 3,
 November 2001, pp. 507-519.

\bibitem[P1]{P1} G. Pag\`es, \emph{A space quantization method for numerical integration}, J. Comput. Appl. Math. 89, 1-38 (1998).
\bibitem[P2]{P2} G. Pag\`es and J. Printems, \emph{Functional quantization for numerics with an application to option pricing}, Monte Carlo Methods
Appl. 11, 407-446 (2005).


\bibitem[R1]{R1} M.K. Roychowdhury, \emph{Quantization and centroidal Voronoi tessellations for probability measures on dyadic Cantor sets}, Journal of Fractal Geometry, 4 (2017), 127-146.


\bibitem[R2]{R2} M.K. Roychowdhury, \emph{Optimal quantizers for some absolutely continuous probability measures}, Real Analysis Exchange, Vol. 43(1), 2017, pp. 105-136.
\bibitem[R3]{R3} M.K.Roychowdhury, \emph{Optimal quantization for the Cantor distribution generated by infinite similitudes},  Israel Journal of Mathematics 231 (2019), 437-466.



\bibitem[R4]{R4} M.K. Roychowdhury, \emph{Least upper bound of the exact formula for optimal quantization of some uniform Cantor distributions}, Discrete and Continuous Dynamical Systems- Series A, Volume 38, Number 9, September 2018, pp. 4555-4570.

\bibitem[R5]{R5} M.K. Roychowdhury, \emph{Center of mass and the optimal quantizers for some continuous and discrete uniform distributions}, to appear, Journal of Interdisciplinary Mathematics.
\bibitem[RR]{RR} J. Rosenblatt and M.K. Roychowdhury, \emph{Optimal quantization for piecewise uniform distributions}, Uniform Distribution Theory 13 (2018), no. 2, 23-55.

 \bibitem[Z]{Z} P.L. Zador, \emph{Development and Evaluation of Procedures for Quantizing Multivariate Distributions}, Ph.D. Thesis (Stanford
University, 1964).



\end{thebibliography}
\end{document}